\documentclass[12pt,a4paper,twoside]{amsart}
\pdfoutput=1
\usepackage{amssymb}
\usepackage{esint}
\usepackage{mathtools}
\usepackage{hyperref}
\usepackage{verbatim}
\usepackage[english]{babel}
\usepackage[latin1]{inputenc}
\usepackage[pdftex]{graphicx}
\theoremstyle{plain}
\newtheorem{theorem}{Theorem}[section]
\newtheorem{lemma}[theorem]{Lemma}
\newtheorem{corollary}[theorem]{Corollary}

\theoremstyle{definition}
\newtheorem{definition}[theorem]{Definition}

\theoremstyle{remark}
\newtheorem{remark}[theorem]{Remark}
\newcommand{\R}{\mathbb R}

\newcommand{\der}{\mathrm{d}}
\newcommand{\eps}{\varepsilon}
\renewcommand{\phi}{\varphi}
\newcommand{\abs}[1]{\left\lvert #1 \right\rvert}

\newcommand{\sisus}{\operatorname{int}}

\renewcommand{\theta}{\vartheta}

\newcommand{\ip}[2]{\left\langle#1,#2\right\rangle}

\title[X-ray tomography for piecewise constant functions]
{Geodesic X-ray tomography for piecewise constant functions on nontrapping manifolds}
\author{Joonas Ilmavirta}
\thanks{Department of Mathematics and Statistics, University of Jyv\"askyl\"a}
\address{Department of Mathematics and Statistics, University of
Jyv\"askyl\"a, P.O. Box 35 (MaD) FI-40014 University of Jyv\"askyl\"a,
Finland}
\email{joonas.ilmavirta@jyu.fi}
\author{Jere Lehtonen}
\email{jere.ta.lehtonen@jyu.fi}
\author{Mikko Salo}
\email{mikko.j.salo@jyu.fi}
\date{\today}

\begin{document}

\begin{abstract}
We show that on a two-dimensional compact nontrapping manifold with strictly convex boundary, a piecewise constant function is determined by its integrals over geodesics. In higher dimensions, we obtain a similar result if the manifold satisfies a foliation condition. These theorems are based on iterating a local uniqueness result. Our proofs are elementary.
\end{abstract}

\maketitle

\section{Introduction}

If~$(M,g)$ is a compact Riemannian manifold with boundary and~$f$ is a function on~$M$, the geodesic X-ray transform of~$f$ encodes the integrals of~$f$ over all geodesics between boundary points.
This transform generalizes the classical X-ray transform that encodes the integrals of a function in Euclidean space over all lines.
The geodesic X-ray transform is a central object in geometric inverse problems and it arises in seismic imaging applications, inverse problems for PDEs such as the Calder\'on problem, and geometric rigidity questions including boundary and scattering rigidity (see the surveys~\cite{PaternainSaloUhlmann2014_survey, Uhlmann2014_BullMathScisurvey}).

It has been conjectured that the geodesic X-ray transform on compact nontrapping Riemannian manifolds with strictly convex boundary is injective~\cite{PSU}.
Here, we say that a Riemannian manifold~$(M,g)$ has \emph{strictly convex boundary} if the second fundamental form of~$\partial M$ in~$M$ is positive definite, and is \emph{nontrapping} if for any $x \in M$ and any $v \in T_x M \setminus \{0\}$ the geodesic starting at~$x$ in direction~$v$ meets the boundary in finite time.

In this work we prove the following result, which verifies this conjecture on two-dimensional nontrapping manifolds if we restrict our attention to piecewise constant functions (see definition~\ref{def:pw-constant}).

\begin{theorem}
\label{thm:main}
Let~$(M,g)$ be a compact nontrapping Riemannian manifold with strictly convex smooth boundary, and let $f\colon M \to \R$ be a piecewise constant function. Let either 
\begin{enumerate}
\item[(a)] $\dim(M) = 2$, or 
\item[(b)] $\dim(M) \geq 3$ and~$(M,g)$ admits a smooth strictly convex function.
\end{enumerate}
If~$f$ integrates to zero over all geodesics joining boundary points, then $f \equiv 0$.
\end{theorem}

This result follows from theorem~\ref{thm:xrt-hd}; for other corollaries, see section~\ref{sec:cor}.
We point out that theorem~\ref{thm:main} implies that a piecewise constant function is uniquely determined by the data if the tiling of the domain is known a priori, but not in general; see remark~\ref{rmk:pw-vectorspace}.

The result for $\dim(M) = 2$ appears to be new, but when $\dim(M) \geq 3$ this is a special case of the much more general result~\cite{UhlmannVasy} that applies to functions in~$L^2(M)$ (see the survey~\cite{PaternainSaloUhlmann2014_survey} for further results on the injectivity of the geodesic X-ray transform also in two dimensions). However, our proof in the case of piecewise constant functions is elementary.
We first prove a local injectivity result (lemmas~\ref{lma:xrt-2d-local} and~\ref{lma:xrt-hd-local}) showing that if a piecewise constant function integrates to zero over all short geodesics near a point where the boundary is strictly convex, then the function has to vanish near that boundary point.
We then iterate the local result by using a foliation of the manifold by strictly convex hypersurfaces as in~\cite{UhlmannVasy, PaternainSaloUhlmannZhou}, given by the existence of a strictly convex function.
The two-dimensional result follows since for $\dim(M) = 2$, the nontrapping condition implies the existence of a strictly convex function~\cite{BeteluGulliverLittman}.
See~\cite{PaternainSaloUhlmannZhou} for further discussion on strictly convex functions and foliations.

\subsection*{Acknowledgements}

J.I.\ was supported by the Academy of Finland (decision 295853).
J.L.\ and M.S.\ were supported by the Academy of Finland (Centre of Excellence in Inverse Problems Research), and M.S.\ was also partly supported by an ERC Starting Grant (grant agreement no 307023).

\section{Preliminaries}

In this section we will define what we mean by regular simplices, piecewise constant functions
and foliations.

\subsection{Regular tilings}

Recall that the standard $n$-simplex is the convex hull of the $n+1$ coordinate unit vectors in~$\R^{n+1}$.

\begin{definition}[Regular simplex]
\label{def:simplex}
Let~$M$ be a manifold with or without boundary, with a $C^1$-structure.
A regular $n$-simplex on~$M$ is an injective $C^1$-image of the standard $n$-simplex in~$\R^{n+1}$.
The embedding is assumed~$C^1$ up to the boundary of the standard simplex.
\end{definition}

The boundary of a regular $n$-simplex is a union of $n+1$ different regular $(n-1)$-simplices.
The boundary is $C^1$-smooth except where the $(n-1)$-simplices intersect.
When discussing interiors and boundaries of simplices, we refer to the natural structure of simplices, not to the toplogy of the underlying manifold.
Every regular $n$-simplex is a manifold with corners in the sense of~\cite{Joyce:OnManifoldsWithCorners}.

\begin{definition}[Depth of a point in a regular simplex]
\label{def:depth}
We associate to each point in a regular simplex an integer which we call the depth of the point.
Interior points have depth~$0$, the interiors of the regular $(n-1)$-simplices making up the boundary have depth~$1$, the interiors of their boundary simplices of dimension $n-2$ have depth~$2$, and so on.
Finally the $n+1$ corner points have depth~$n$.
\end{definition}

\begin{definition}[Regular tiling]
\label{def:tiling}
Let~$M$ be a manifold with or without boundary, with a $C^1$-structure.
A regular tiling of~$M$ is a collection of regular $n$-simplices~$\Delta_i$, $i\in I$, so that the following hold:
\begin{enumerate}
\item The collection is locally finite: for any compact set $K\subset M$ the index set $\{i\in I;\Delta_i\cap K\neq\emptyset\}$ is finite.
  (Consequently, if~$M$ is compact, the index set~$I$ itself is necessarily finite.)
\item $M=\bigcup_{i\in I}\Delta_i$.
\item $\sisus(\Delta_i)\cap\sisus(\Delta_j)=\emptyset$ when $i\neq j$.
\item If $x\in\Delta_i\cap\Delta_j$, then~$x$ has the same depth in both~$\Delta_i$ and~$\Delta_j$.
\end{enumerate}
\end{definition}

The simplices of a tiling have boundary simplices, the boundary simplices have boundary simplices, and so forth.
We refer to all these as the boundary simplices of the tiling.

\subsection{Tangent cones of simplices}
\label{sec:tangent-cone-def}

We will need tangent spaces of simplices, and at corners these are more naturally conical subsets than vector subspaces of the tangent spaces of the underlying manifold.

\begin{definition}[Tangent cone and tangent space of a regular simplex]
\label{def:tangent}
Let~$M$ be a manifold with or without boundary, with a $C^1$-structure.
Consider a regular $m$-simplex~$\Delta$ in~$M$ with $0\leq m\leq n=\dim(M)$.
Let $x\in\Delta$, and let $\Gamma = \Gamma(x,\Delta)$ be the set of all $C^1$-curves starting at~$x$ and staying in~$\Delta$.
The tangent cone of~$\Delta$ at~$x$, denoted by~$C_x\Delta$, is the set
\begin{equation}
\{\dot\gamma(0);\gamma\in\Gamma\}\subset T_xM.
\end{equation}
The tangent space of~$\Delta$ at~$x$, denoted by~$T_x\Delta$, is the vector space spanned by $C_x\Delta\subset T_xM$.
\end{definition}

One can easily verify that for any $m$-dimensional regular simplex~$\Delta$ the tangent cone~$C_x\Delta$ is indeed a closed subset of~$T_xM$ and~$T_x\Delta$ is the tangent space in the usual sense.
The cone is scaling invariant but it need not be convex.
If~$x$ is an interior point of~$\Delta$, then $C_x\Delta=T_x\Delta$ and they coincide with~$T_xM$ if $m=n$.

\subsection{Piecewise constant functions}

We are now ready to give a definition of piecewise constant functions and their tangent functions.

\begin{definition}[Piecewise constant function]
\label{def:pw-constant}
Let~$M$ be a manifold with or without boundary, with a $C^1$-structure.
We say that a function $f\colon M\to\R$ is piecewise constant if there is a regular tiling $\{\Delta_i;i\in I\}$ of~$M$ so that~$f$ is constant in the interior of each regular $n$-simplex~$\Delta_i$ and vanishes on $\bigcup_{i\in I}\partial\Delta_i$.
\end{definition}

The assumption of vanishing on the union of lower dimensional submanifolds is not important; we choose it for convenience.
The values of the function in this small set play no role in our results.

In dimension two one can essentially equivalently define piecewise constant functions via tilings by curvilinear polygons.
The only difference is in values on lower dimensional manifolds.

\begin{definition}[Tangent function]
\label{def:tangent-function}
Let~$M$ be a manifold with or without boundary, with a $C^1$-structure.
Let~$f$ be a piecewise constant function on it and~$x$ any point in~$M$.
Let $\Delta_1,\dots,\Delta_N$ be the simplices of the tiling that contain~$x$.
Denote by $a_1,\dots,a_N$ the constant values of~$f$ in the interiors of these simplices.
The tangent function $T_xf\colon T_xM\to\R$ of~$f$ at~$x$ is defined so that for each $i\in\{1,\dots,N\}$ the function~$T_xf$ takes the constant value~$a_i$ in the interior of the tangent cone~$C_x\Delta_i$ (see definition~\ref{def:tangent}).
The tangent function takes the value zero outside $\bigcup_{i=1}^N \sisus (C_x\Delta_i)$.
\end{definition}

If~$x$ is an interior point of a regular simplex in the tiling (it has depth zero in the sense of definition~\ref{def:depth}), then the tangent function is a constant function with the constant value of the ambient simplex.

\begin{remark}
\label{rmk:pw-vectorspace}
The set of all piecewise constant functions $M\to\R$ is not a vector space, since the intersection of two regular simplices is not always a union of regular simplices. Thus, if~$f_1$ and~$f_2$ are piecewise constant functions that have the same integrals over geodesics, then theorem~\ref{thm:main} shows that $f_1 = f_2$ in~$M$ whenever~$f_1$ and~$f_2$ are adapted to a common regular tiling but not in general.
\end{remark}

\subsection{Foliations}

For foliations we use the definition given in~\cite{PaternainSaloUhlmannZhou}:
\begin{definition}[Strictly convex foliation]
\label{def:foliation}
Let~$M$ be a smooth Riemannian manifold with boundary.
\begin{enumerate}
  \item The manifold~$M$ satisfies the foliation condition if there is a smooth strictly convex function $\phi \colon M \to \R$.
  \item A connected open subset~$U$ of~$M$ satisfies the foliation condition if there is a smooth strictly convex
  exhaustion function $\phi \colon U \to \R$, in the sense that the set $\{x \in U; \phi(x) \geq c\}$ is compact for any
  $c > \inf_U \phi$.\label{def:foliation:subset}
\end{enumerate}
\end{definition}

If (\ref{def:foliation:subset}) is satisfied, then $U \cap \partial M \neq \emptyset$, the level sets
of~$\phi$ provide a foliation of~$U$ by smooth strictly convex hypersurfaces (except possibly
at the minimum point of~$\phi$ if $U = M$), and the fact that~$\phi$ is an exhaustion function
ensures that one can iterate a local uniqueness result to obtain uniqueness in all of~$U$ by a layer stripping argument.

Furthermore, since the set $\{x \in U; \phi(x) \geq c\}$ is compact for any $c > \inf_U \phi$,
it follows that, intuitively speaking,
the level sets $\{\phi=c\}$ extend all the way from~$\partial M$ to~$\partial M$ without terminating at~$\partial U$. For more details, see~\cite{PaternainSaloUhlmannZhou}.

\section{X-ray tomography of conical functions in the Euclidean plane}

Let us denote the upper half plane by $H_+=\{(x,y)\in\R;y>0\}$.
Let $\alpha_1>\alpha_2>\dots>\alpha_N>\alpha_{N+1}$ and $a_1,\dots,a_N$ be any real numbers.
Consider the function $f\colon H_+\to\R$ given by
\begin{equation}
\label{eq:plane-cone}
f(x,y)
=
\begin{cases}
a_1, & \alpha_1y > x > \alpha_2y\\
a_2, & \alpha_2y > x > \alpha_3y\\
\vdots\\
a_N, & \alpha_Ny > x > \alpha_{N+1}y,\\
\end{cases}
\end{equation}
and $f(x,y) = 0$ for other $(x,y) \in H_+$.
This is an example of a tangent function (see definition~\ref{def:tangent-function}) of a piecewise constant function.
The analysis of this archetypical example is crucial to the proof of our main result in all dimensions.

Fix some $h>0$ and consider the lines~$\ell_t$ given by $y=h+tx$.

\begin{lemma}
\label{lma:2d-corner-uniqueness}
Fix any $h>0$, an integer $N\geq1$ and real numbers $\alpha_1>\alpha_2>\dots>\alpha_{N+1}$.
Let $f\colon H_+\to\R$ be as in~\eqref{eq:plane-cone} above.
Then the numbers $a_1,\dots,a_N$ are uniquely determined by the integrals of~$f$ over the lines~$\ell_t$ where~$t$ ranges in any neighborhood of zero.
\end{lemma}

\begin{proof}
It suffices to show that if~$f$ integrates to zero over all these lines, then all the constants $a_1,\dots, a_N$ are zero.

The lines~$\ell_t$ and $x=\alpha_iy$ intersect at a point whose first coordinate is $\frac{h\alpha_i}{1-\alpha_it}\eqqcolon hz_i^t$.
The integral of~$f$ over the line~$\ell_t$ is
\begin{equation}
\begin{split}
\int_{\ell_t}f \der s
&=
h\sqrt{1+t^2}[
(z_1^t-z_2^t)a_1+
(z_2^t-z_3^t)a_2+
\dots\\
&\quad+
(z_N^t-z_{N+1}^t)a_N
].
\end{split}
\end{equation}
Since this vanishes for all~$t$ near zero, we get
\begin{equation}
c_1^ta_1+
c_2^ta_2+
\dots+
c_N^ta_N
=
0
\end{equation}
for all~$t$, where $c_i^t=z_i^t-z_{i+1}^t$.

Let~$D_t^k$ denote the~$k$th order derivative with respect to~$t$.
Differentiating with respect to~$t$ gives the system of equations
\begin{equation}
\begin{cases}
c_1^ta_1+\dots+c_N^ta_N&=0\\
D_tc_1^ta_1+\dots+D_tc_N^ta_N&=0\\
\frac12D_t^2c_1^ta_1+\dots+\frac12D_t^2c_N^ta_N&=0\\
\vdots\\
\frac1{(N-1)!}D_t^{N-1}c_1^ta_1+\dots+\frac1{(N-1)!}D_t^{N-1}c_N^ta_N&=0.
\end{cases}
\end{equation}
We will show that this system is uniquely solvable for the numbers~$a_i$ at $t=0$.

Since
\begin{equation}
z_i^t
=
\frac{\alpha_i}{1-\alpha_it}
=
\alpha_i \left[1+\alpha_it+(\alpha_it)^2+(\alpha_it)^3+\dots\right],
\end{equation}
we easily observe
\begin{equation}
\left.\frac1{k!}D_t^kc_i^t\right|_{t=0}
=
\alpha_i^{k+1}-\alpha_{i+1}^{k+1}.
\end{equation}
Therefore our system of equations at $t=0$ takes the form
\begin{equation}
A
\begin{pmatrix}
a_1\\a_2\\\vdots\\a_N
\end{pmatrix}
=
0,
\end{equation}
where
\begin{equation}
\label{eq:vdm}
A
=
\begin{pmatrix}
\alpha_1-\alpha_2 & \alpha_2-\alpha_3 & \cdots & \alpha_N-\alpha_{N+1} \\
\alpha_1^2-\alpha_2^2 & \alpha_2^2-\alpha_3^2 & \cdots & \alpha_N^2-\alpha_{N+1}^2 \\
\vdots & \vdots & \ddots & \vdots \\
\alpha_1^N-\alpha_2^N & \alpha_2^N-\alpha_3^N & \cdots & \alpha_N^N-\alpha_{N+1}^N
\end{pmatrix}.
\end{equation}
We need to show now that the matrix~$A$ is invertible.
This will be proven in lemma~\ref{lma:matrix} below, and that concludes the proof of the present lemma.
\end{proof}

\begin{lemma}
\label{lma:matrix}
The determinant of the matrix~$A$ in~\eqref{eq:vdm} is
\begin{equation}
(-1)^N \, \prod_{\mathclap{1\leq i<j\leq N+1}} \, (\alpha_j-\alpha_i),
\end{equation}
which is non-zero if all numbers~$\alpha_i$ are distinct.
\end{lemma}

\begin{proof}
We modify the matrix following these steps:
\begin{itemize}
\item Add the last column to the second last one, then the second last one to the third last one, and so forth.
Finally add the second column to the first one.
The determinant is not changed.
\item Make the transformation
\begin{equation}
A
\mapsto
\begin{pmatrix}
0 & 1 \\
A & v
\end{pmatrix}
.
\end{equation}
The determinant is multiplied by~$(-1)^N$.
We can choose the vector~$v$ freely, and we pick $v=(\alpha_{N+1},\alpha_{N+1}^2,\dots,\alpha_{N+1}^N)$.
\item
Add the last column to all other columns.
The determinant is unchanged.
\end{itemize}
We end up with the matrix
\begin{equation}
\begin{pmatrix}
1 & 1 & \cdots & 1 \\
\alpha_1 & \alpha_2 & \cdots & \alpha_{N+1} \\
\vdots & \vdots & \ddots & \vdots \\
\alpha_1^N & \alpha_2^N & \cdots & \alpha_{N+1}^N
\end{pmatrix}.
\end{equation}
This is a Vandermonde matrix and the determinant is well known to be the product of differences.
\end{proof}

\section{Limits of geodesics at corners in two dimensions}

\subsection{Geodesics on manifolds and tangent spaces}
\label{sec:geodesic-intro}

Let~$M$ be a $C^2$-smooth Riemannian surface with $C^2$-boundary.
Suppose the boundary~$\partial M$ is strictly convex at $x\in\partial M$.
Let~$\gamma_i$, $i=1,2$, be two unit speed $C^1$-curves in~$M$ starting from~$x$ so that the initial velocities~$\dot\gamma_i(0)$ are distinct and non-tangential.

Let $r>0$ be such that $B(x,r)\subset M$ is split by the curves into three parts.
Let~$A$ be the middle one.

Let~$\sigma_i$, $i=1,2$, be the curves on~$T_xM$ with constant speed~$\dot\gamma_i(0)$, respectively.
Let~$S$ be the sector in~$T_xM$ lying between~$\sigma_1$ and~$\sigma_2$, corresponding to~$A$.

If $h>0$ and if $v\in T_xM$ is an inward pointing unit vector, let the geodesic~$\gamma^h_v$ be constructed as follows:
Take a unit vector~$w$ normal to~$v$ at~$x$ and let~$w(h)$ be its parallel transport along the geodesic through~$v$ by distance~$h$.
Then take the maximal geodesic in this direction.
For sufficiently small~$h$ and~$v$ sufficiently close to normal the geodesic~$\gamma^h_v$ has endpoints near~$x$ since the boundary is strictly convex.

Moreover, let~$\sigma^h_v$ be the similarly constructed line on~$T_xM$.
Varying~$h$ translates the line and varying~$v$ rotates it.

\subsection{Limits in two dimensions}

We are now ready to present our key lemma~\ref{lma:key} which allows us to reduce the problem into a Euclidean one.

\begin{lemma}
\label{lma:key}
Let~$M$ be a $C^2$-smooth Riemannian surface with $C^2$-boundary, and assume the boundary is strictly convex at $x\in\partial M$.
For an open set of vectors~$v$ in a neighborhood of the inward unit normal, we have
\begin{equation}
\label{eq:sector_limit}
\lim_{h\to0}\frac1h\int_{\gamma^h_v\cap A}\der s
=
\int_{\sigma^1_v\cap S}\der s,
\end{equation}
where we have used the notation of section~\ref{sec:geodesic-intro}.
\end{lemma}

Notice that $\frac1h\int_{\sigma^h_v\cap S}\der s$ is independent of $h>0$ by scaling invariance.

\begin{proof}[Proof of lemma~\ref{lma:key}]
We will prove the claim for $v=\nu$, the inward pointing unit normal to the boundary. The same argument will work if~$v$ is sufficiently close to normal (so that neither~$\dot\gamma_i(0)$ is parallel to~$w$).

Since the statement is local we can assume that we are working in~$\R^2$ near the boundary point~$0$, and the metric~$g$ is extended smoothly outside~$M$. We wish to express the left hand side of \eqref{eq:sector_limit} in terms of the intersection points of~$\gamma_i$ with~$\gamma_v^h$. To do this, consider the map $F\colon B \to \R^{2+2}$,
\begin{equation}
F(h, s) = (\gamma_1(s_1) - \gamma_v^h(s_3), \gamma_2(s_2) - \gamma_v^h(s_4)),
\end{equation}
where $B \subset \R^{1+4}$ is a small neighborhood of the origin and $s = (s_1, s_2, s_3, s_4)$.
The map~$F$ is~$C^1$ and satisfies $F(0,0) = 0$.
The $s$-derivatives are given by 
\begin{equation}
D_{s} F(0,0)
=
\begin{pmatrix}
\dot\gamma_1(0) & 0 & -w & 0 \\
0 & \dot\gamma_2(0) & 0 & -w
\end{pmatrix}
\end{equation}
and this matrix is invertible since~$\dot{\gamma}_i(0)$ are nontangential. By the implicit function theorem, there is a $C^{1}$-function~$s(h)$ near~$0$ so that 
\begin{equation}
F(h, s)
=
0 \text{ for $(h,s)$ near $0$} \quad \iff \quad s = s(h).
\end{equation}
We may now express the quantity of interest as 
\begin{equation}
\lim_{h\to0}\frac1h\int_{\gamma^h_v\cap A}\der s = \lim_{h\to0}\frac{s_4(h)-s_3(h)}{h} = s'_4(0) - s'_3(0).
\end{equation}

It remains to compute~$s'_3(0)$ and~$s'_4(0)$.
Differentiating the equation $F(h, s(h)) = 0$ with respect to~$h$ gives that 
\begin{equation}
s'(0)
=
- D_{s} F(0,0)^{-1} \partial_h F(0,0).
\end{equation} 
Since $\gamma_v^h(0) = \eta(h)$ where~$\eta(h)$ is the geodesic through~$v$ we have $\partial_h \gamma_v^h(0)|_{h=0} = v$ and $\partial_h F(0, 0) = (-v, -v)$. Then by direct computations
\begin{equation}
s'(0)
=
\begin{pmatrix}
  \frac{w^2 v^1 - w^1 v^2}{\dot{\gamma}_1^1(0) w^2 - \dot{\gamma}_1^2(0) w^1} \\[5pt]
  \frac{w^2 v^1 - w^1 v^2}{\dot{\gamma}_2^1(0) w^2 - \dot{\gamma}_2^2(0) w^1} \\[5pt]
  \frac{\dot{\gamma}_1^2(0) v^1 - \dot{\gamma}_1^1(0) v^2}{\dot{\gamma}_1^1(0) w^2 - \dot{\gamma}_1^2(0) w^1} \\[5pt]
  \frac{\dot{\gamma}_2^2(0) v^1 - \dot{\gamma}_2^1(0) v^2}{\dot{\gamma}_2^1(0) w^2 - \dot{\gamma}_2^2(0) w^1}
\end{pmatrix},
\end{equation}
where $v = (v^1,v^2)$, $w = (w^1,w^2)$, and 
$\dot\gamma_{i}(0) = (\dot{\gamma}_i^1(0),\dot{\gamma}_i^2(0))$.

We assume that~$\nu$ is between~$\dot\gamma_1(0)$ and~$\dot\gamma_2(0)$ (similar reasoning works also in the other cases). Notice that
\begin{equation}
  s'_3(0)
  = \frac{- \sin(\angle(\dot\gamma_1(0),v))}{\sin(\angle(\dot\gamma_1(0),w))}
  \quad\text{and}\quad
  s'_4(0)
  = \frac{-\sin(\angle(\dot\gamma_2(0),v))}{\sin(\angle(\dot\gamma_2(0),w))}.
\end{equation}
Next we calculate $\int_{\sigma^1_v\cap S}\der s$ for $v=\nu$.
We denote by $z_i \in T_x M$ the intersection point
of the curves~$\sigma_i$ and~$\sigma_\nu^1$.
Since~$\nu$ is between $\dot\sigma_1(0) = \dot\gamma_1(0)$ and $\dot\sigma_2(0) = \dot\gamma_2(0)$,
we find $\abs{z_1-\nu} = -s'_3(0)$ and $\abs{\nu - z_2} = s'_4(0)$.
Thus
\begin{equation}
  \int_{\sigma^1_v\cap S}\der s = \abs{z_1-\nu} + \abs{\nu - z_2} = s'_4(0) - s'_3(0). \qedhere
\end{equation}
\end{proof}

We can extend lemma~\ref{lma:key} to piecewise constant functions and their tangent functions.

\begin{lemma}
\label{lma:sector-limit}
Let~$M$ be a $C^2$-smooth Riemannian surface with $C^2$-boundary, and assume the boundary is strictly convex at $x\in\partial M$. 
Let $\tilde M\supset M$ be an extension so that~$x$ is an interior point of~$\tilde{M}$.
Let $f\colon\tilde M\to\R$ be a piecewise constant function and assume that~$T_xf$ is supported in an inward-pointing cone which meets~$T_x\partial M$ only at $0\in T_xM$.


For all vectors~$v$ in some neighborhood of the inward unit normal $\nu\in T_xM$, we have
\begin{equation}
\lim_{h\to0}\frac1h\int_{\gamma^h_{v}}f\der s
=
\int_{\sigma^1_{v}}T_xf\der s,
\end{equation}
where we have used the notation of section~\ref{sec:geodesic-intro} and definition~\ref{def:tangent-function}.
Here the geodesics~$\gamma^h_v$ are geodesics of~$M$ and do not extend into $\tilde M\setminus M$.
\end{lemma}

\begin{proof}
It suffices to apply lemma~\ref{lma:key} to each cone~$C_x\Delta$ of a simplex~$\Delta$ containing~$x$ separately.
\end{proof}

We remark that the extension~$\tilde M$ only plays a role in the tiling, not in the geodesics.

\section{X-ray tomography in two dimensions}

\subsection{Local result}
\label{sec:xrt-2d-local}
Let us suppose that~$M$ is a Riemannian surface and~$x$ is a point in its interior.
Suppose~$\Sigma$ is a hypersurface going through the point~$x$ and that~$\Sigma$ is strictly convex in a neighborhood of~$x$. If~$V$ is a sufficiently
small neighborhood of~$x$ then $V \setminus \Sigma$ consists of two open sets which we denote by~$V_+$
and~$V_-$. Here~$V_+$ is the open set for which the part of the boundary coinciding with~$\Sigma$ is strictly convex.

\begin{lemma}\label{lma:xrt-2d-local}
Let~$M$ be a $C^2$-smooth Riemannian surface
and assume that $f\colon M\to\R$ is a piecewise constant function in the sense of definition~\ref{def:pw-constant}.
Fix $x \in \sisus(M)$ and let~$\Sigma$ be a 1-dimensional hypersurface (curve) through~$x$.
Suppose that~$V$ is a neighborhood of~$x$ so that
\begin{itemize}
  \item $V$ intersects only simplices containing~$x$,
  \item $\Sigma$ is strictly convex in~$V$,
  \item $f|_{V_-} = 0$, and
  \item $f$ integrates to zero over every maximal geodesic
    in~$V$ having endpoints on~$\Sigma$.
\end{itemize}
Then $f|_V = 0$.
\end{lemma}

\begin{remark}\label{rmk:xrt-2d-local}
We will also use the lemma in the case where $x \in \partial M$, the boundary is strictly convex at~$x$, and $\Sigma = \partial M$ (the set~$V_-$ is not needed then). The proof of the lemma is valid also in this situation.
\end{remark}

\begin{figure}
\includegraphics[width=12cm]{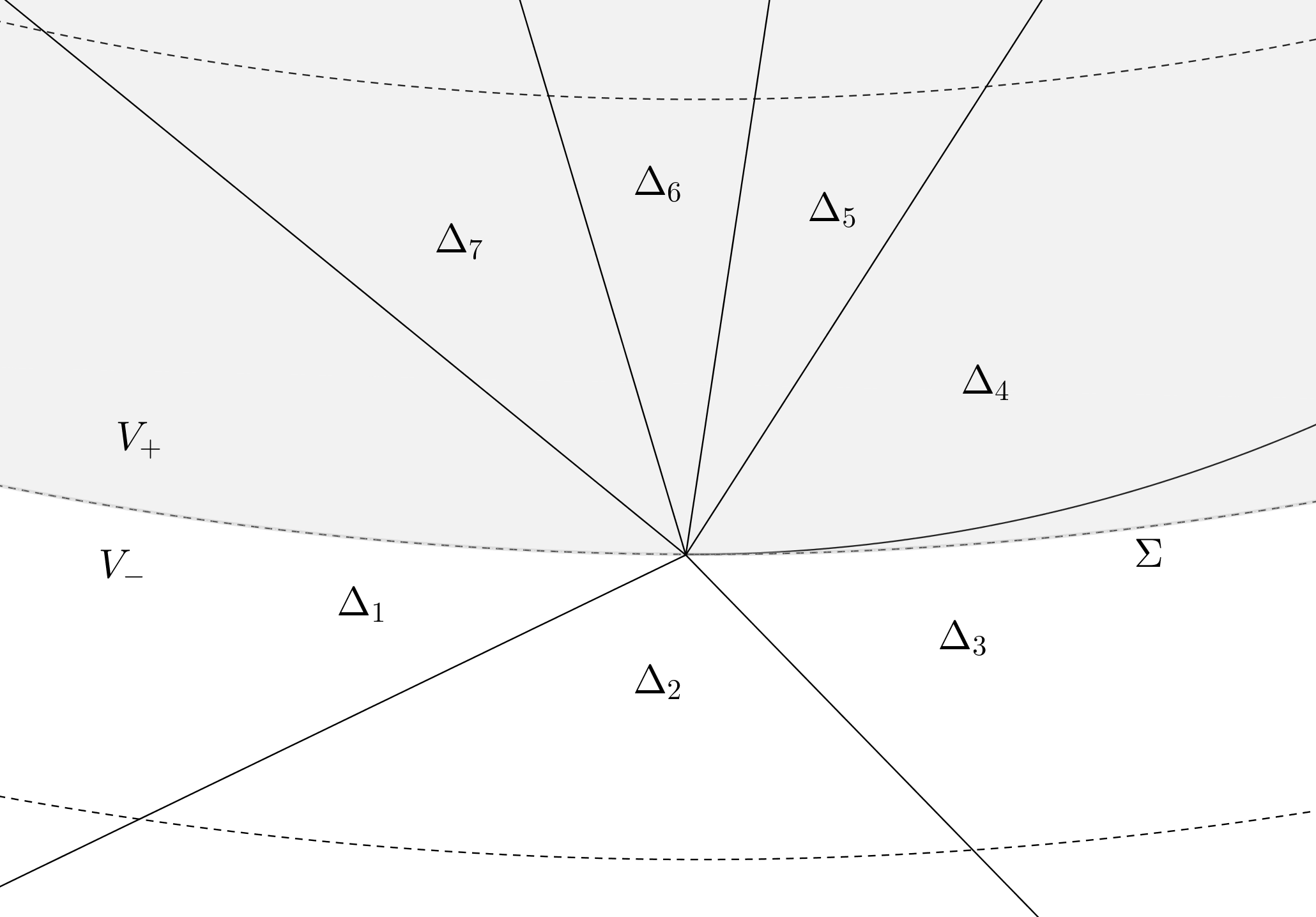}
\caption{The geometric setting of lemma~\ref{lma:xrt-2d-local}.
Simplices $\Delta_1,\Delta_2,\Delta_3$ are of the first type,
simplex~$\Delta_4$ is of the second type,
and simplices $\Delta_5,\Delta_6,\Delta_7$ are of the third type.
Later, this lemma will be applied in the case where~$\Sigma$ is the level
of an exhaustion function, and the dashed lines represent level sets of this function.}
\label{fig:lma:xrt-2d-local}
\end{figure}

\begin{proof}[Proof of lemma~\ref{lma:xrt-2d-local}]
We denote by $\Delta_1,\dots,\Delta_N$ the simplices containing the point~$x$.
The case $N=1$ is trivial, so we suppose that $N > 1$. Assume that~$\nu$ is the normal of~$\Sigma$ at~$x$ pointing into~$V_+$.
We denote $H_\pm = \{ w \in T_x M ; \pm \ip{\nu}{w} > 0 \}$
and $H_0 = T_x\Sigma \subset T_x M$.

We divide simplices $\Delta_1, \dots, \Delta_N$ into
three mutually exclusive types as follows:
\begin{enumerate}
  \item simplices~$\Delta$ so that $C_x \Delta \cap H_- \neq \emptyset$,
  \item simplices~$\Delta$ so that $C_x \Delta \subset H_+ \cup H_0$ and $C_x \Delta \cap H_0 \neq \{0\}$, and 
  \item simplices~$\Delta$ so that $C_x \Delta \subset H_+ \cup \{0\}$.
\end{enumerate}
Different types of simplices are illustrated in figure~\ref{fig:lma:xrt-2d-local}.

Let us first suppose that simplex~$\Delta$ is of the first type.
Since~$C_x\Delta$ has a non-empty interior, the set $C_x \Delta \cap H_-$ has a non-empty interior.
It must be that $\Delta \cap V_- \neq \emptyset$ and
thus~$f$ vanishes on~$\Delta$.

Suppose then that~$\Delta$ is a simplex of the second type.
We take $v \in C_x \Delta \cap H_0$, with $v \neq 0$, and define~$\gamma_v^\eps$ to be the geodesic
with initial data $\gamma_v^\eps(0) = x$ and $\dot{\gamma}_v^\eps(0) = v + \eps \nu$
where $\eps > 0$ and~$\nu$ points into~$V_+$. Since~$\Sigma$ is strictly convex near~$x$,
we can take~$\eps$ to be small enough so that the geodesic has endpoints on $\Sigma \cap V$
and it is contained in~$V$.

For small positive~$t$ we have $\gamma_v^\eps(t) \in \Delta$.
If~$\gamma_v^\eps$ is completely contained in~$\Delta$ we immediately
get that $f|_\Delta = 0$.
If this is not the case then there is a unique simplex~$\tilde\Delta$
so that $C_x \Delta \cap C_x \tilde\Delta = C_x\Delta \cap H_0 = C_x \tilde\Delta \cap H_0$,
in other words simplices~$\Delta$ and~$\tilde\Delta$ have a common boundary simplex
which is tangent to~$\Sigma$ at~$x$.
Thus for small~$\eps$ we have $\gamma_v^\eps \subset \Delta \cup \tilde\Delta$.
Since tangent cones of simplices have non-empty
interior it must be that $\tilde\Delta \cap V_- \neq \emptyset$,
which implies that $f|_{\tilde\Delta} = 0$.
By our assumptions~$f$ integrates to zero over~$\gamma_v^\eps$, hence $f|_\Delta = 0$.

We are left with simplices of the third type. For those we can apply lemma~\ref{lma:sector-limit}
combined with lemma~\ref{lma:2d-corner-uniqueness}: The
geodesics~$\gamma_v^h$ introduced in section~\ref{sec:geodesic-intro}
are contained in~$V$ for small~$h$ and~$v$ close enough
to normal. By lemma~\ref{lma:sector-limit} integrals of~$T_x f$ are zero over
the lines~$\sigma_v^1$ for~$v$ close enough to normal.
Since~$f$ can be non-zero only in simplices
of the third type we can apply lemma~\ref{lma:2d-corner-uniqueness}
to conclude that $T_xf = 0$ which implies that~$f$ vanishes
on those simplices too.
\end{proof}

\subsection{Global result}
The local result of lemma~\ref{lma:xrt-2d-local} combined
with the foliation condition allows us to obtain a global result:

\begin{theorem}
\label{thm:xrt-2d}
Let~$M$ be a smooth Riemannian surface with strictly convex boundary.
Suppose there is a strictly convex foliation of an open subset $U\subset M$ in the sense of definition~\ref{def:foliation}.
Let $f\colon M\to\R$ be a piecewise constant function in the sense of definition~\ref{def:pw-constant}.
If~$f$ integrates to zero over all geodesics in~$U$, then $f|_U=0$.
\end{theorem}

\begin{remark}
In the previous result, it is not required to assume that~$\partial M$ is strictly convex (it would actually follow from the foliation condition that the foliation starts at a strictly convex boundary point). However, we have made this assumption for convenience.
\end{remark}

\begin{proof}[Proof of theorem~\ref{thm:xrt-2d}]
Denote $T=\max_U\phi$.
The sets $U_t = \{ \phi \geq t \}$ are compact for every $t > \inf_U \varphi$ by assumption,
and $\bigcup_{t > \inf_U \varphi} U_t = U$. It suffices to show that $f|_{U_t} = 0$ for any $t > \inf_U \varphi$.

Fix any such~$t$. The set~$U_t$ meets only finitely many regular simplices $\Delta_1,\dots,\Delta_N$ of the tiling corresponding to~$f$.
Denote by $T_1 > T_2 > \dots > T_k$ the distinct elements of the set $\{ \max_{\Delta_i}\phi; 1 \leq i \leq N\}$.
Note that $T_1 = T$.

First, take any point $x\in U\cap\partial M$ for which $\phi(x)=T$.
By remark~\ref{rmk:xrt-2d-local} the function~$f$ vanishes on a neighborhood of~$x$.
Therefore~$f$ vanishes on all simplices~$\Delta_i$ for which $\max_{\Delta_i}\phi = T = T_1$,
and hence in the set $\{\phi>T_2\}$.

We wish to continue this argument at points of the level set $\{ \phi = T_2 \}$. We can apply lemma~\ref{lma:xrt-2d-local}
at all points of $\{\phi=T_2\}$ which are in~$\sisus(M)$ to show that~$f$ vanishes near these points. This uses the fact that~$f$ will integrate to zero along short geodesics in $\{ \phi \leq T_2 \}$ near such points, since the foliation condition implies that the maximal extensions of such short geodesics reach~$\partial M$ in finite time by~\cite[Lemma 6.1]{PaternainSaloUhlmannZhou} and since~$f$ integrates to zero over geodesics in~$U$. The points in $\{\phi=T_2\}$ which are not in~$\sisus(M)$ are handled similarly by using remark~\ref{rmk:xrt-2d-local} (such points are on~$\partial M$, since the set $\{\phi \geq T_2\}$ cannot intersect the boundary of~$U$ except at the boundary of~$M$).
We find that~$f$ vanishes on all simplices~$\Delta_i$ for which $\max_{\Delta_i}\phi = T_2$.
Continuing iteratively we reach the index~$k$ and conclude that~$f$ vanishes on all simplices $\Delta_1,\dots,\Delta_N$.
\end{proof}

We remark that the two-dimensional versions of the corollaries presented in section~\ref{sec:cor} follow from theorem~\ref{thm:xrt-2d}.

\section{Higher dimensions}
\label{sec:hd}

\subsection{Local result} As in the two-dimensional case we begin by proving a local result.
First we need a technical lemma.

\begin{lemma}
\label{lma:good-2-plane}
Let~$M$ be an $n$-dimensional $C^1$-smooth Riemannian manifold with or without boundary and $\{\Delta_i;i\in I\}$ a regular tiling of it.
Take $x\in M$ of depth at least~$1$ and~$\Sigma$ a hypersurface through it.
Let $\Delta_1,\dots,\Delta_N$ be the simplices meeting~$x$.
For any $i\in\{1,\dots,N\}$ there is a $2$-plane $P \subset T_xM$
so that the following hold:
\begin{enumerate}
  \item $P\cap\sisus(C_x\Delta_i)\neq\emptyset$,
  \item for any boundary simplex~$\delta$ of dimension $n-2$ or lower in the tiling, we have $P\cap C_x\delta=\{0\}$, and
  \item $\dim(P \cap T_x\Sigma) = 1$.
\end{enumerate}
\end{lemma}

\begin{proof}
We begin by picking a vector~$v$ from a small neighborhood of~$\nu$, where~$\nu$ is any unit normal of~$\Sigma$ at~$x$, so that $v \in \sisus(C_x \Delta_j)$ for some~$j$.
Let $\pi_v \colon T_xM\to T_x\Sigma$ be the projection down to~$T_x \Sigma$ in direction of~$v$, i.e., $\pi_v(av + z) = z$ whenever $a \in \R$ and $z \in T_x \Sigma$.
Let~$\delta$ be any regular $m$-simplex for $m\leq n-2$ contained in the tiling as a boundary simplex.
The projection~$\pi_v(T_x\delta)$ has dimension~$m$ (since $v \in \sisus(C_x \Delta_j)$) and therefore codimension $n-1-m\geq1$ in~$T_x\Sigma$.
There are finitely many such simplices~$\delta$,
so there is an open dense subset of vectors in~$\pi_v(C_x \Delta_i)$
that do not belong to the $\pi_v$-projection of the tangent cone of any boundary simplex of dimension
$n-2$ or lower.
We pick a vector~$w$ from that set and 
take~$P$ to be the plane spanned by~$v$ and~$w$. This~$P$ satisfies all the requirements.
\end{proof}

The next lemma is a higher dimensional analogue of lemma~\ref{lma:xrt-2d-local}.
Here~$V_+$ and~$V_-$ are defined similarly as in the beginning of section~\ref{sec:xrt-2d-local}.

\begin{lemma}\label{lma:xrt-hd-local}
Let~$M$ be a $C^2$-smooth Riemannian manifold 
and $f\colon M\to\R$ be a piecewise constant function in the sense of definition~\ref{def:pw-constant}.
Fix $x \in \sisus(M)$ and let~$\Sigma$ be a $(n-1)$-dimensional hypersurface through~$x$.
Suppose that~$V$ is a neighborhood of~$x$ so that
\begin{itemize}
  \item $V$ intersects only simplices containing~$x$,
  \item $\Sigma$ is strictly convex in~$V$,
  \item $f|_{V_-} = 0$, and
  \item $f$ integrates to zero over every maximal geodesic
    in~$V$ having endpoints on~$\Sigma$.
\end{itemize}
Then $f|_V = 0$.
\end{lemma}

\begin{remark}
As in the two-dimensional case this lemma holds also for points
of the boundary with minor modifications. See remark~\ref{rmk:xrt-2d-local}.
\end{remark}

\begin{proof}[Proof of lemma~\ref{lma:xrt-hd-local}]
We denote by $\Delta_1,\dots,\Delta_N$ the simplices containing the point~$x$.
The case $N=1$ is trivial, so we suppose that $N > 1$.
We denote by~$H_+$ and~$H_-$ the open upper half-plane and the open lower half-plane in~$T_x M$ corresponding to~$V_+$ and~$V_-$.
Furthermore we denote $H_0 = T_xM \setminus (H_+ \cup H_-)$.

As in the two-dimensional case (lemma~\ref{lma:xrt-2d-local}) we divide
simplices $\Delta_1,\dots,\Delta_N$ into three mutually exclusive
types resembling those of the two-dimensional case:
\begin{enumerate}
  \item simplices~$\Delta$ so that $C_x \Delta \cap H_- \neq \emptyset$,
  \item simplices~$\Delta$ so that $C_x \Delta \subset H_+ \cup H_0$
    and $\dim (C_x\Delta \cap H_0) = n-1$ (i.e.\ $C_x\Delta \cap H_0$ contains an open set of $H_0$), and
  \item simplices~$\Delta$ so that $C_x \Delta \subset H_+ \cup H_0$ and
    $\dim (C_x\Delta \cap H_0) \leq n-2$.
\end{enumerate}

The proof that~$f$ vanishes on simplices of first and second type is the same as in the two-dimensional case.

Let~$\Delta$ be a simplex of the third type
and~$P$ be a plane given by lemma~\ref{lma:good-2-plane}
so that $P \cap \sisus(C_x\Delta) \neq \emptyset$.
We define $P_+ = P \cap H_+$, and~$P_-$ and~$P_0$ similarly.

We wish to show that $T_xf|_P = 0$. This implies that $f|_\Delta = 0$.
Since~$\Delta$ is arbitrary we can conclude that $T_x f = 0$ or,
equivalently, $f|_V = 0$.
To prove $T_x f|_P = 0$, we use a two-dimensional argument similar
to that used in the proof of lemma~\ref{lma:xrt-2d-local}.
In order to do that we must first show that~$T_x f|_{P}$
vanishes outside some closed conical set in~$P_+$.

Suppose that~$\tilde\Delta$ is a simplex so that $C_x \tilde\Delta \cap (P_0 \cup P_-) \neq \{0\}$.
We want to show that it is of the first or second type.
If $C_x \tilde\Delta \cap H_- \neq \emptyset$, meaning that the simplex is of the first type, we have $f|_{\tilde\Delta} = 0$.
If this is not the case, then $C_x \tilde\Delta \subset H_+ \cup H_0$ and $C_x \tilde{\Delta} \cap P_0 \neq \{0\}$.
Furthermore we know that for any boundary simplex~$\delta$
of~$\tilde \Delta$ for which $\dim(\delta)\leq n-2$ we have $P \cap \delta = \{ 0\}$.
Thus every vector in $P_0 \cap C_x \tilde\Delta$ must be in the interior of a tangent
cone of some $(n-1)$-dimensional boundary simplex.
Therefore $\dim (C_x\tilde\Delta \cap H_0) = n-1$, so the simplex~$\tilde \Delta$ is of the second type and hence
$f|_{\tilde\Delta} = 0$.

The remaining case is where $C_x \tilde{\Delta} \subset P_+ \cup \{0\}$. To deal with this case, we will apply lemma~\ref{lma:sector-limit} on suitably chosen submanifolds.
We choose $w_0 \in T_x \Sigma \cap P = P_0$ and $v_0 \in P_+$ orthogonal to~$w_0$.
Then~$P$ is spanned by~$v_0$ and~$w_0$.
Suppose $v \in P$ is in a neighborhood of~$v_0$ and 
$w \in P$ is perpendicular to~$v$.
We discussed the geodesics~$\gamma_{v}^h$ on~$M$ and~$\sigma_{v}^h$ on~$T_xM$ in section~\ref{sec:geodesic-intro}.
In two dimensions they did not depend on the choice of~$w$ (except for sign).
In higher dimensions they do, and we will denote the corresponding curves in $V_+ \cup \Sigma$ by~$\gamma_{v,w}^h$ and the tangent space curves by~$\sigma_{v,w}^h$ instead.

The family of geodesics $\{\gamma_{v,w}^h;h\in[0,h_0)\}$
foliates a smooth two-dimensional manifold $M_{v,w}\subset V_+ \cup \Sigma$, for a small enough $h_0>0$.
The geodesics~$\gamma_{v,w}^h$ are also geodesics on the manifold~$M_{v,w}$ although the submanifold may not be totally geodesic. Note that if~$M_{v_0,w_0}$ happens to be totally geodesic, which is always the case when $n=2$, then $M_{v,w} = M_{v_0,w_0}$ for all such~$v$ and~$w$.

The foliated manifold~$M_{v,w}$ has four essential properties:
its boundary near~$x$ (which is a subset of~$\Sigma$) is strictly convex, we have $T_x M_{v,w} = P$, the vector~$v_0$
is the inward pointing boundary normal at~$x$, and for small enough~$h_0$ the tiling
of~$M$ induces a proper tiling for~$M_{v,w}$ near~$x$.
To see this, observe that the plane~$P$ meets all $(n-1)$-dimensional boundary simplices transversally and does not meet lower dimensional simplices outside the origin, so the surface~$M_{v,w}$ will locally do the same due to the implicit function theorem.

We construct~$\sigma_{v,w}^h$ on~$P$ as in section~\ref{sec:geodesic-intro}.
Now we can apply lemma~\ref{lma:sector-limit} on manifold~$M_{v,w}$ 
to get $\int_{\sigma^1_{v,w}} T_xf|_P \, \der s = 0$.
By varying~$v$ and~$w$, and hence varying~$M_{v,w}$ too,
we are able to reduce the problem to a Euclidean one on the tangent space~$P$. We can
apply lemma~\ref{lma:2d-corner-uniqueness}
to deduce that $T_xf|_P = 0$. Especially $f|_\Delta = 0$.
\end{proof}

\subsection{The key theorem}

We next present our key theorem in all dimensions.
Theorem~\ref{thm:xrt-hd} contains theorem~\ref{thm:xrt-2d}.

\begin{theorem}
\label{thm:xrt-hd}
Let~$M$ be a smooth Riemannian manifold with strictly convex boundary.
Assume $\dim(M)\geq2$.
Suppose there is a strictly convex foliation of an open subset $U\subset M$ in the sense of definition~\ref{def:foliation}.
Let $f\colon M\to\R$ be a piecewise constant function in the sense of definition~\ref{def:pw-constant}.
If~$f$ integrates to zero over all geodesics in~$U$, then $f|_U=0$.
\end{theorem}

\begin{proof}
The proof is very similar to the two-dimensional case given in theorem~\ref{thm:xrt-2d}.
The analogue of lemma~\ref{lma:xrt-2d-local} for higher dimensions is provided by lemma~\ref{lma:xrt-hd-local}.
The rest of the proof is unchanged.
\end{proof}

\subsection{Corollaries}
\label{sec:cor}

Let us discuss some consequences of theorem~\ref{thm:xrt-hd}.

The theorem can be seen as a support theorem.
In particular, the classical support theorem of Helgason~\cite{HelgasonBook} for the X-ray transform in the case of piecewise constant functions follows easily from our theorem.

Corollaries~\ref{cor:many-foliations} and~\ref{cor:one-foliation} are easy to prove from theorem~\ref{thm:xrt-hd}.

\begin{corollary}
\label{cor:many-foliations}
Let~$M$ be a smooth Riemannian manifold with strictly convex boundary.
Let $\dim(M)\geq2$.
Suppose open subsets $U_1,\dots,U_N\subset M$ each have a foliation in the sense of definition~\ref{def:foliation}.
If a piecewise constant function $f \colon M\to\R$ has zero X-ray transform, then it vanishes in $\bigcup_{i=1}^NU_i$.
\end{corollary}

In particular, in the case $N=1$ and $U_1=M$ we find:

\begin{corollary}
\label{cor:one-foliation}
Let~$M$ be a smooth Riemannian manifold with strictly convex boundary and with a strictly convex foliation.
Assume $\dim(M)\geq2$.
If a piecewise constant function $f \colon M\to\R$ has zero X-ray transform, then it vanishes everywhere.
\end{corollary}

If we do not assume the foliation condition, but instead just assume the boundary to be strictly
convex locally, lemma~\ref{lma:xrt-hd-local} implies the following result, which
can be seen as a local support theorem:

\begin{corollary}
\label{cor:no-foliation-local}
Let~$M$ be a $C^2$-smooth Riemannian manifold with boundary.
Assume $\dim(M)\geq2$ and let $f \colon M\to\R$ be a piecewise constant function.
Suppose that $x \in \partial M$ is such that the boundary of~$M$ is strictly convex at~$x$.
If~$f$ integrates to zero over geodesics having endpoints near~$x$,
then~$f$ vanishes in a neighborhood of~$x$.

Especially if the boundary of~$M$ is strictly convex and~$f$ integrates
to zero over all geodesic contained in a neighborhood of the boundary,
then~$f$ vanishes near the boundary.
\end{corollary}

If $\dim(M) \geq 3$ the previous result is a special case of the local support theorem for $L^2$-functions proven in~\cite{UhlmannVasy}. For $\dim(M) = 2$, a similar support theorem is known on real-analytic simple surfaces~\cite{Krishnan}.

\subsection{Proof of theorem~\ref{thm:main}}
Part (a) follows from corollary~\ref{cor:one-foliation} and the fact that a compact nontrapping Riemannian surface with strictly convex boundary always has a strictly convex foliation, see~\cite{BeteluGulliverLittman, PaternainSaloUhlmannZhou}.
Part (b) is implied by corollary~\ref{cor:one-foliation}. The theorem is proven.

\bibliographystyle{abbrv}
\bibliography{piecewise-constant}

\end{document}